\documentclass{amsart}
\title{Stable Centres of Iwahori-Hecke Algebras of type A}
\author{Christopher Ryba}
\address{Department of Mathematics, University of California, Berkeley, CA 94720, USA}
\email{ryba@math.berkeley.edu}
\usepackage{hyperref}
\usepackage{fullpage}
\usepackage{amsmath}
\usepackage{amsfonts}
\usepackage{amssymb}
\usepackage{amsthm}
\usepackage{mathrsfs}  
\usepackage{ytableau}
\usepackage{subfig}
\usepackage{graphicx, caption}
\usepackage{comment}
\usepackage{bm}
\usepackage{todonotes}
\usepackage{float}

\usepackage{tikz-cd}

\DeclareMathOperator{\Id}{Id}

\newcommand{\FH}{\mathrm{FH}}

\newtheorem{theorem}{Theorem}[section]
\newtheorem{lemma}[theorem]{Lemma}
\newtheorem{proposition}[theorem]{Proposition}
\newtheorem{corollary}[theorem]{Corollary}
\newtheorem{definition}[theorem]{Definition}
\newtheorem{remark}[theorem]{Remark}
\newtheorem{example}[theorem]{Example}

\begin{document}
\maketitle
\begin{abstract}
A celebrated result of Farahat and Higman constructs an algebra $\mathrm{FH}$ which ``interpolates'' the centres $Z(\mathbb{Z}S_n)$ of group algebras of the symmetric groups $S_n$. We extend these results from symmetric group algebras to type $A$ Iwahori-Hecke algebras, $H_n(q)$. In particular, we explain how to construct an algebra $\mathrm{FH}_q$ ``interpolating'' the centres $Z(H_n(q))$. We prove that $\mathrm{FH}_q$ is isomorphic to $\mathcal{R}[q,q^{-1}] \otimes_{\mathbb{Z}} \Lambda$ (where $\mathcal{R}$ is the ring of integer-valued polynomials, and $\Lambda$ is the ring of symmetric functions). The isomorphism can be described as ``evaluation at Jucys-Murphy elements'', leading to a proof of a conjecture of Francis and Wang. This yields character formulae for the Geck-Rouquier basis of $Z(H_n(q))$ when acting on Specht modules. 
\end{abstract}

\section{Introduction}
\noindent
Let $H_n(q)$ be the Iwahori-Hecke algebra of type $A_{n-1}$, which may be viewed as a $q$-deformation of the group algebra of the symmetric group $S_n$. In this paper we extend stability results about centres of symmetric group algebras to centres of Iwahori-Hecke algebras. When $q=1$, so that $H_n(q)$ specialises to $\mathbb{Z}S_n$, we recover prior results of Farahat-Higman \cite{FarahatHigman}, Corteel-Goupil-Schaeffer \cite{CorteelGoupilSchaeffer}, and the author \cite{Ryba}. 
\newline \newline \noindent
Let us describe these results in the more familiar symmetric group case. The centre $Z(\mathbb{Z}S_n)$ has a basis consisting of conjugacy-class sums. Farahat and Higman proved that the multiplicative structure constants of this basis are evaluations of certain polynomials, and used these polynomials to define an algebra $\mathrm{FH}$ which canonically surjects onto $Z(\mathbb{Z}S_n)$ for any $n$ \cite{FarahatHigman}. It turns out that $\mathrm{FH}$ is isomorphic to $\mathcal{R} \otimes_{\mathbb{Z}} \Lambda$ where $\mathcal{R}$ is the ring of integer-valued polynomials and $\Lambda$ is the ring of symmetric functions. Moreover, this isomorphism can be interpreted as evaluation at Jucys-Murphy (``JM'') elements (see Section 3 of \cite{Ryba}). The relationship between JM elements and Gelfand-Zetlin bases of Specht modules leads to a short proof of the Nakayama Conjecture which determines the $p$-blocks of the symmetric group. One also obtains formulae for central characters of symmetric groups in terms of contents of partitions \cite{CorteelGoupilSchaeffer}.
\newline \newline \noindent
In the present paper, we explain how to extend all these facts about symmetric groups to Iwahori-Hecke algebras. Although $H_n(q)$ is not manifestly a group algebra, its centre still has a basis analogous to conjugacy class sums: the Geck-Rouquier basis. M\'{e}liot proved a polynomiality property for the multiplicative structure constants of the Geck-Rouquier basis which generalises the symmetric group case \cite{Meliot_Hecke}. This allows us to define an algebra $\mathrm{FH}_q$, which canonically surjects onto $Z(H_n(q))$ for each $n$. We check that $\mathrm{FH}_q$ degenerates to $\mathrm{FH}$ when $q=1$. Also we show that $\mathrm{FH}_q$ is isomorphic to $\mathcal{R}[q,q^{-1}] \otimes_{\mathbb{Z}} \Lambda$ via evaluation at JM elements, which proves a conjecture of Francis and Wang \cite{FrancisWang}. We use this to prove character formulae for Geck-Rouquier elements of Iwahori-Hecke algebras in terms of $q$-contents of partitions. Along the way, we reprove the characterisation of blocks of $H_n(q)$ using $\mathrm{FH}_q$.
\newline \newline \noindent
The structure of the paper is as follows. In Section \ref{sec:background}, we review properties of the ring of symmetric functions, $\Lambda$, as well as the Farahat-Higman algebra $\mathrm{FH}$. In Section \ref{section:Hecke_alg}, we discuss Iwahori-Hecke algebras and the Geck-Rouquier basis of $Z(H_n(q))$. In Section \ref{sec:FHq} we define the main object of the paper, $\mathrm{FH}_q$, check that it specialises to $\mathrm{FH}$ when $q=1$, and show that it is isomorphic to $\mathcal{R}[q,q^{-1}] \otimes_{\mathbb{Z}} \Lambda$ via JM evaluation. Finally in Section \ref{section:applications}, we explain how this theory determines the blocks of the Iwahori-Hecke algebra (over a field of arbitrary characteristic) and how we obtain character formulae for Geck-Rouquier basis elements in terms of the contents of partitions.
\newline \newline \noindent
\textbf{Acknowledgements.} The author would like to thank Weiqiang Wang for helpful comments.

\section{Background} \label{sec:background}
\subsection{Symmetric Functions}
\noindent
We recall some generalities about partitions and the ring of symmetric functions, $\Lambda$; everything we need can be found in Chapter 1 of \cite{Macdonald}.
\newline \newline \noindent
Recall that a partition is a sequence of non-negative integers $\lambda = (\lambda_1, \lambda_2, \ldots, \lambda_l)$ which is non-increasing. We do not distinguish between partitions that differ only by trailing zeros. The $\lambda_i$ are called the parts of $\lambda$. The size of a partition $\lambda$ is $|\lambda| = \lambda_1 + \lambda_2 + \cdots + \lambda_l$, and its length $l(\lambda)$ is the number of nonzero parts of $\lambda$.
\newline \newline \noindent
There is an action of $S_n$ on $\mathbb{Z}[x_1, x_2, \ldots, x_n]$ by permuting the polynomial variables. We write $\Lambda_n = \mathbb{Z}[x_1, x_2, \ldots, x_n]^{S_n}$ for the invariants with respect to this action. So we may think of $\Lambda_n$ as polynomials that are symmetric in $n$ variables. If $m > n$, there is a map $\rho_{m,n}: \Lambda_m \to \Lambda_n$ which sets the the variables $x_{n+1}, x_{n+2}, \ldots, x_{m}$ to zero. These maps define an inverse system, and the inverse limit in the category of graded rings is the ring of symmetric functions, $\Lambda$. It is convenient to view $\Lambda$ as the set of ``polynomials'' in infinitely many variables $x_1, x_2, \ldots$ that are symmetric in these variables, such as $e_1 = x_1 + x_2 + \cdots$. (The use of the term ``polynomial'' is not strictly correct since such expressions are not finite sums.)
\newline \newline \noindent
By formal properties of inverse limits, $\Lambda$ inherits a map to each $\Lambda_n$ which amounts to setting the variables $x_{n+1}, x_{n+2}, \ldots$ to zero. We obtain a genuine symmetric polynomial in $n$ variables, which we may evaluate at a sequence of $n$ numbers. Since the polynomial is symmetric, the order of the terms in the sequence does not matter. In Section \ref{section:applications}, we will make use of this fact, and write $f(S)$ to mean the evaluation of the symmetric function $f$ at the $n$-element multiset $S$.
\newline \newline \noindent
There are two families of symmetric functions that appear in this paper. One family consists of the monomial symmetric functions, $m_\lambda$ ($\lambda$ a partition). We define $m_\lambda$ to be the sum of all monomials $x_1^{a_1} x_2^{a_2} \cdots$, where the exponents $a_i$ yield the partition $\lambda$ when sorted from largest to smallest. The $m_\lambda$ form a $\mathbb{Z}$-basis of $\Lambda$. The second family are the elementary symmetric functions $e_r = m_{(1^r)}$, which are the coefficients of the following generating function:
\[
\sum_{r \geq 0} e_r t^r = \prod_{i \geq 1}(1 + x_i t).
\]
We will use the fact that $\Lambda$ is the free polynomial algebra in the elementary symmetric functions: $\Lambda = \mathbb{Z}[e_1, e_2, \ldots]$. Note that $m_\lambda$ is homogeneous of degree $|\lambda|$ and $e_r$ is homogeneous of degree $r$.

\subsection{The Farahat-Higman algebra}
\noindent
We review the classical results of Farahat and Higman \cite{FarahatHigman}. For a more detailed discussion including applications to modular representation theory, see Section 3 of \cite{Ryba}.
\newline \newline \noindent
Conjugacy classes of symmetric groups $S_n$ are labelled by partitions of $n$ via cycle type. In order to compatibly label conjugacy classes of symmetric groups of different sizes, (so that, for example, transpositions have the same label regardless of which symmetric group they belong to), we use the following definition.

\begin{definition}
The \emph{reduced cycle type} of an element of a symmetric group is the partition obtained by subtracting $1$ from each part of the cycle type.
\end{definition}
\noindent
The identity element of $S_n$ has cycle type $(1^n)$, but its reduced cycle type is the empty partition, $\varnothing$. Similarly, a transposition in $S_n$ has cycle type $(2,1^{n-1})$ and reduced cycle type $(1)$. There are elements of reduced cycle type $\mu$ in $S_n$ precisely when $n \geq |\mu| + l(\mu)$. The reduced cycle type allows us to easily work with all symmetric groups at once. Let $X_\mu$ be the sum of all elements of $S_n$ of reduced cycle type $\mu$; $X_\mu$ is either the sum of all elements in a conjugacy class, or zero if $S_n$ has no elements of reduced cycle type $\mu$. 
\begin{example} \label{eqn:FH_alg_example}
We have the equation
\[
X_{(1)}^2 = 2X_{(1,1)} + 3X_{(2)} + \binom{n}{2}X_{\varnothing}.
\]
This equation hold in $Z(\mathbb{Z}S_n)$ for any $n \in \mathbb{Z}_{\geq 0}$. 
\end{example}
\noindent
What is essential is that the coefficients in the Example \ref{eqn:FH_alg_example} are polynomials in $n$. This turns out to hold for arbitrary products of $X_\mu$.

\begin{definition}
The \emph{ring of integer-valued polynomials}, $\mathcal{R}$, is the subring of $\mathbb{Q}[t]$ consisting of polynomials $p(t)$ such that $p(m) \in \mathbb{Z}$ for all $m \in \mathbb{Z}$.
\end{definition}
\noindent
A standard fact is that $\mathcal{R}$ consists of $\mathbb{Z}$-linear combinations of binomial coefficients $\binom{t}{r}$ (in partucular, $\mathcal{R}$ is free as a $\mathbb{Z}$-module).

\begin{theorem}[Farahat-Higman, Theorem 2.2 \cite{FarahatHigman}] \label{FH_symmetric_group_polynomials}
For each triple of partitions $\lambda, \mu, \nu$, there is a unique integer-valued polynomial $\phi_{\mu, \nu}^\lambda(t) \in \mathcal{R}$ such that
\[
X_\mu X_\nu = \sum_\lambda \phi_{\mu, \nu}^\lambda(n)X_\lambda
\]
in $Z(\mathbb{Z}S_n)$ for all $n \in \mathbb{Z}_{\geq 0}$.
\end{theorem}
\noindent
This allows us to ``interpolate'' $Z(\mathbb{Z}S_n)$ with respect to the parameter $n$.

\begin{definition}[Farahat-Higman, Section 2 \cite{FarahatHigman}] \label{FH_alg_definition}
The \emph{Farahat Higman algebra}, $\mathrm{FH}$, is the $\mathcal{R}$-algebra defined as follows. It is free as a $\mathcal{R}$-module with basis $K_\mu$ indexed by partitions $\mu$, and is equipped with a bilinear multiplication defined by
\[
K_\mu K_\nu = \sum_{\lambda} \phi_{\mu, \nu}^\lambda(t) K_\lambda,
\]
where $\phi_{\mu, \nu}^\lambda(t) \in \mathcal{R}$ are the polynomials appearing in Theorem \ref{FH_symmetric_group_polynomials}.
\end{definition}
\noindent
Theorem \ref{FH_symmetric_group_polynomials} says that when we evaluate the polynomial variable $t$ at $n$, the $K_{\mu}$ have the same structure constants as $X_{\mu}$. So we obtain a canonical map $\mathrm{FH} \to Z(\mathbb{Z}S_n)$.

\begin{theorem}[Farahat-Higman, Theorem 2.4 \cite{FarahatHigman}] \label{thm:FH_spec_hom}
For each $n \in \mathbb{Z}_{\geq 0}$ there is a surjective ring homomorphism $\Phi_n: \FH \to Z(\mathbb{Z}S_n)$ defined by
\[
\Phi_n \left( \sum_{\mu} a_\mu(t) K_\mu \right) = \sum_{\mu} a_\mu(n) X_\mu,
\]
where $a_\mu(t) \in \mathcal{R}$, and $X_\mu$ is the sum of all elements of $S_n$ of reduced cycle type $\mu$ (or zero, if there are no such elements).
\end{theorem}
\noindent
By using the maps $\Phi_n$, one can deduce properties of $\mathrm{FH}$ from those of $Z(\mathbb{Z}S_n)$. In particular, one can easily show that $\FH$ is a commutative, associative, unital $\mathcal{R}$-algebra.

\begin{remark}
A different approach to $\mathrm{FH}$ was given by Ivanov and Kerov \cite{IvanovKerov}. Their approach can be extended from symmetric groups to other settings. In particular, M\'{e}liot \cite{Meliot_Hecke} adapted this approach to prove a version of Theorem \ref{FH_symmetric_group_polynomials} for Iwahori-Hecke algebras (see Theorem \ref{thm:FH_hecke_polys} below). Another generalisation is to group algebras of finite classical groups in work of Kannan and the author \cite{KannanRyba}.
\end{remark}
\noindent
It turns out that $\mathrm{FH}$ is isomorphic go $\mathcal{R} \otimes_{\mathbb{Z}} \Lambda$, and moreover this isomorphism may be interpreted as evaluation at JM elements. We summarise this, and direct the reader to Section 3 of \cite{Ryba} for more details. The Jucys-Murphy (JM) elements $L_i$ in $\mathbb{Z}S_n$ are defined as a sum of transpositions:
\[
L_i = \sum_{j < i} (i,j).
\]
It is well-known that these elements commute pairwise, and moreover any symmetric function $f \in \Lambda$ evaluated at $L_1, L_2, \ldots, L_n$ yields a central element of $\mathbb{Z}S_n$. In the case $f = e_r$, it takes a convenient form.

\begin{theorem}[Jucys, Section 3 \cite{Jucys}] \label{thm:elem_JM}
We have
\[
e_r(L_1, \ldots, L_n) = \sum_{\mu \vdash r} X_\mu,
\]
where $e_r$ is the $r$-th elementary symmetric function, and $X_\mu$ is the sum of elements of reduced cycle type $\mu$ in $S_n$.
\end{theorem}
\noindent
The dependence on $n$ in Theorem \ref{thm:elem_JM} appears on the left hand side by the number of JM elements appearing, and on the right in the ambient group algebra $\mathbb{Z}S_n$. To instead make a statement independent of $n$, we make the following definition.
\begin{definition}
Let $\Psi: \mathcal{R} \otimes_{\mathbb{Z}} \Lambda \to \mathrm{FH}$ be the $\mathcal{R}$-algebra homomorphism defined by
\[
\Psi(e_r) = \sum_{\mu \vdash r} K_\mu.
\]
\end{definition}
\noindent
As $\Lambda$ is freely generated as a polynomial algebra (over $\mathbb{Z}$) by the elementary symmetric functions, there are no relations to be checked to guarantee that $\Psi$ is a homomorphism.
\begin{theorem}[Theorem 3.8 \cite{Ryba}] \label{thm:FH_lambda_iso}
The map $\Psi$ is an isomorphism from $\mathcal{R} \otimes_{\mathbb{Z}} \Lambda$ to $\mathrm{FH}$.
\end{theorem}
\noindent
Since the definition of $\Psi$ ``interpolates'' evaluation of symmetric functions at JM elements, $\FH$ can be thought of as being ``the ring of symmetric functions evaluated at JM elements''.

\section{The Iwahori-Hecke Algebra} \label{section:Hecke_alg}
\noindent
We aim to prove an Iwahori-Hecke-algebra version of Theorem \ref{thm:FH_lambda_iso}, but this requires some preparation. For example, the correct analogue of conjugacy-class sums in the symmetric group algebra is the Geck-Rouquier basis (of the centre of the Iwahori-Hecke algebra). There is a notion of Jucys-Murphy element, and symmetric polynomials in them yield central elements. However, it was not until relatively recently \cite{FrancisGraham} that symmetric polynomials in the JM elements were proved to span the centre of the Iwahori-Hecke algebra (in the case of the symmetric group, it was show in in \cite{Murphy}). For general background on Iwahori-Hecke algebras we direct the reader to \cite{Mathas}.

\begin{definition}
The \emph{type $A$ Iwahori-Hecke algebra}, $H_n(q)$, is generated over $\mathbb{Z}[q,q^{-1}]$ by $T_i$ ($1 \leq i \leq n-1$), where $(T_i - q)(T_i+1) = 0$, and the $T_i$ satisfy the braid relations: $T_i T_{i+1} T_i = T_{i+1} T_i T_{i+1}$ ($1 \leq i \leq n-2$) and $T_i T_j = T_j T_i$ (when $|i-j| \geq 2$).
\end{definition}
\noindent
Specialising $q$ to $1$, we obtain the Coxeter presentation of $\mathbb{Z}S_n$, the integral group algebra of the symmetric group $S_n$, where $T_i$ is identified with the adjacent transposition $s_i = (i,i+1)$. Thus the Iwahori-Hecke algebra may be thought of as a $q$-deformation of the symmetric group algebra. There are other conventions for the definition of the Iwahori-Hecke algebra, but this one is most convenient for our purposes.
\newline \newline \noindent
It is well known that $H_n(q)$ has a $\mathbb{Z}[q,q^{-1}]$-basis $T_w$ indexed by $w \in S_n$ defined as follows. A reduced expression for $w$ is an expression of the form $w = s_{i_1}s_{i_2} \cdots s_{i_l}$ where each $s_{i_k}$ is an adjacent transposition $(i_k, i_k+1)$, with $l$ as small as possible (the \emph{length} of $w$). Then we define $T_w = T_{i_1} T_{i_2} \cdots T_{i_l}$. Matsumoto's theorem asserts that any two reduced expressions for $w \in S_n$ may be obtained from each other by applying the braid relations appropriately. Since the same relations hold for the Iwahori-Hecke algebra generators $T_i$, applying the braid relations in the same way shows that the corresponding expressions for $T_w$ also coincide.

\subsection{Geck-Rouquier Basis}
To adapt the concept of the Farahat-Higman algebra to the setting of Iwahori-Hecke algebras, we need distinguished elements of the centre $Z(H_n(q))$ to play the role of conjugacy-class sums in $Z(\mathbb{Z}S_n)$. The correct elements are the Geck-Rouquier basis $\Gamma_\mu$ of $Z(H_n(q))$, which are a $q$-deformation of the conjugacy-class sums, in the sense that specialising $q$ to $1$ turns $\Gamma_\mu$ to the conjugacy-class sum $X_\mu$. (Note that we use reduced cycle types to index central elements, rather than ordinary cycle types.) 

\begin{theorem}[Geck-Rouquier, Section 5 \cite{GeckRouquier}]
The centre $Z(H_n(q))$ is a free $\mathbb{Z}[q,q^{-1}]$-module with a basis $\Gamma_\mu$ (indexed by partitions $\mu$ obeying $|\mu| + l(\mu) \leq n$) uniquely defined by the following properties:
\begin{enumerate}
\item $\Gamma_\mu$ is a central element of $H_n(q)$,
\item specialising $q$ to $1$ (and identifying $H_n(1)$ with $\mathbb{Z}S_n$) turns $\Gamma_\mu$ into the conjugacy class sum $X_\mu$,
\item let $\Gamma_\mu = \sum_w c_{\mu, w}(q) T_w$ (where $c_{\mu, w}(q) \in \mathbb{Z}[q,q^{-1}]$), and suppose that $w$ is of minimal length within its conjugacy class. Then $a_{\mu, w}(q) = 1$ if $w$ is of reduced cycle type $\mu$, and $a_{\mu,w}(q) = 0$ if $w$ is not of reduced cycle type $\mu$.
\end{enumerate}
\end{theorem}
\noindent
No formula for expressing $\Gamma_\mu$ in the $T_w$ basis of $H_n(q)$ is known.

\begin{example}
In $H_3(q)$, there are three applicable reduced cycle types: $2$, $1$, $\varnothing$. The corresponding Geck-Rouquier basis elements are
\begin{eqnarray*}
\Gamma_\varnothing &=& T_{\Id}, \\
\Gamma_{(1)} &=& T_{(12)} + T_{(23)} + q^{-1}T_{(13)}, \\
\Gamma_{(2)} &=& T_{(123)} + T_{(132)} + (q-1)q^{-1}T_{(13)}.
\end{eqnarray*}
We point out that $(13)$ has length 3, which is not minimal among elements of reduced cycle type $(1)$, i.e. transpositions. Hence the coefficient of $T_{(13)}$ is allowed be different from $0$ or $1$. Note also that $q^{-1}$ appears in these elements, so it is essential that we work over $\mathbb{Z}[q,q^{-1}]$ rather than $\mathbb{Z}[q]$.
\end{example}

\noindent
Adapting ideas of Ivanov and Kerov \cite{IvanovKerov}, M\'{e}liot proved the analogous version of Theorem \ref{FH_symmetric_group_polynomials} for the Geck-Rouquier basis, proving Conjecture 3.1 of \cite{FrancisWang}. Let $\mathcal{R}[q,q^{-1}] = \mathbb{Z}[q,q^{-1}] \otimes_{\mathbb{Z}} \mathcal{R}$.

\begin{theorem}[M\'{e}liot, Theorem 1 \cite{Meliot_Hecke}] \label{thm:FH_hecke_polys}
For any partitions $\lambda, \mu, \nu$, there exists a unique $\phi_{\mu, \nu}^\lambda(q,t) \in \mathcal{R}[q,q^{-1}]$ such that the equation
\[
\Gamma_\mu \Gamma_\nu = \sum_{\lambda} \phi_{\mu, \nu}^\lambda(q,n) \Gamma_\lambda
\]
holds in $Z(H_n(q))$ for all $n \in \mathbb{Z}_{\geq 0}$. (We take $\Gamma_\lambda$ to be zero if $|\lambda| + l(\lambda) > n$.)
\end{theorem}
\begin{example} \label{ex:Hecke_example}
At the end of the paper \cite{Meliot_Hecke}, the following equation is found:
\[
\Gamma_{(1)}^2 = (q + q^{-1}) \Gamma_{(1,1)} + (q+1+q^{-1})\Gamma_{(2)} + (n-1)(q-1)\Gamma_{(1)} + {n \choose 2}q \Gamma_\varnothing.
\]
Setting $q=1$ turns $\Gamma_\lambda$ into $X_\lambda$, and we recover Example \ref{eqn:FH_alg_example}.
\end{example}

\section{\texorpdfstring{The $q$-Farahat-Higman Algebra}{The q-Farahat-Higman Algebra}} \label{sec:FHq}
\noindent 
\subsection{\texorpdfstring{Construction of $\mathrm{FH}_q$}{Construction of FHq}}
As observed in Section 3 of \cite{FrancisWang}, Theorem \ref{thm:FH_hecke_polys} allows us to construct a $q$-deformation of the Farahat-Higman algebra from Definition \ref{FH_symmetric_group_polynomials}.

\begin{definition} \label{def:FH_hecke_alg}
Let $\mathrm{FH}_q$ be the free $\mathcal{R}[q,q^{-1}]$-module with basis $K_{\mu,q}$ indexed by all partitions $\mu$. Define a bilinear multiplication on $\mathrm{FH}_q$ by
\[
K_{\mu,q} K_{\nu,q} = \sum_{\lambda} \phi_{\mu, \nu}^\lambda(q,t) K_{\lambda,q}
\]
where $\phi_{\mu, \nu}^\lambda(q,t) \in \mathcal{R}[q,q^{-1}]$ are as defined by Theorem \ref{thm:FH_hecke_polys}. We call $\mathrm{FH}_q$ the \emph{$q$-Farahat-Higman algebra}.
\end{definition}
\noindent
Similarly to Theorem \ref{thm:FH_spec_hom}, we have an evaluation map allowing us to return to $Z(H_n(q))$ from $\mathrm{FH}_q$.
\begin{proposition} \label{prop:specisalisation_hom}
For each $n \in \mathbb{Z}_{\geq 0}$ there is a surjective $\mathbb{Z}[q,q^{-1}]$-algebra homomorphism $\Phi_{n,q} :\mathrm{FH}_q \to Z(H_n(q))$ defined by
\[
\Phi_{n,q} \left( \sum_{\mu} a_\mu(q,t) K_{\mu,q} \right) = \sum_{\mu} a_\mu(q,n) \Gamma_\mu,
\]
where $a_\mu(q,t) \in \mathcal{R}[q,q^{-1}]$, and $\Gamma_\mu$ is taken to be zero if $|\mu| + l(\mu) > n$.
\end{proposition}
\begin{proof}
Note that $\Phi_{n,q}(K_{\mu,q}) = \Gamma_{\mu}$, and the Geck-Rouquier basis spans $Z(H_n(q))$. Therefore $\Phi_{n,q}$ is surjective. It is a homomorphism by Theorem \ref{thm:FH_hecke_polys}. 
\end{proof}
\noindent
From here it is routine to confirm that $\mathrm{FH}_q$ is commutative, associative, and unital. These properties can be phrased in terms of equations in $\mathcal{R}[q,q^{-1}]$ involving the structure constants $\phi_{\mu, \nu}^\lambda(q,t)$. For example, commutativity is the assertion that $K_{\mu, q}K_{\nu,q} - K_{\nu,q}K_{\mu,q} = 0$, or that
\[
\sum_{\lambda} \left(\phi_{\mu,\nu}^\lambda(q,t) - \phi_{\nu,\mu}^\lambda(q,t) \right) K_{\lambda, q} = 0.
\]
We can apply $\Phi_{n,q}$ to the left hand side, passing to $Z(H_n(q))$:
\[
\sum_{\lambda} \left(\phi_{\mu,\nu}^\lambda(q,n) - \phi_{\nu,\mu}^\lambda(q,n) \right) \Gamma_{\lambda} = \Phi_{n,q}(K_{\mu, q}K_{\nu,q} - K_{\nu,q}K_{\mu,q}) = \Gamma_\mu \Gamma_\nu - \Gamma_\nu \Gamma_\mu = 0,
\]
where in the last step we used the fact that $Z(H_n(q))$ is commutative. Hence if $n \geq |\lambda| + l(\lambda)$ so that $\Gamma_\lambda$ is nonzero, we conclude that $\phi_{\mu,\nu}^\lambda(q,n) - \phi_{\nu,\mu}^\lambda(q,n) = 0$. In particular, we have $\phi_{\mu,\nu}^\lambda(q,t) - \phi_{\nu,\mu}^\lambda(q,t) = 0$ for $t$ in a Zariski-dense set, and hence in general. Similar arguments apply for associativity and unitality.

\begin{proposition} \label{prop:diagram}
The following diagram commutes:
\begin{equation*}
\begin{tikzcd}
\mathrm{FH}_q     \arrow[r, "\Theta"]    \arrow[d, "\Phi_{n,q}"]    &    \FH    \arrow[d, "\Phi_n"]\\
Z(H_n(q)) \arrow[r,"\theta"]&    Z(\mathbb{Z}S_n)
\end{tikzcd}
\end{equation*}
where $\theta: H_n(q) \to \mathbb{Z}S_n$ is evaluation at $q=1$ and $\Theta$ is the map
\[
\Theta \left( \sum_{\mu} a_\mu(q,t) K_{\mu,q} \right) = \sum_{\mu} a_\mu(1,t) K_\mu
\]
which is a ring homomorphism.
\end{proposition}
\begin{proof}
The commutativity of the diagram follows from the fact that both compositions are given by
\[
\sum_{\mu}a_\mu(q,t) K_{\mu, q} \mapsto \sum_{\mu} a_\mu(1,n) X_\mu.
\]
To see that $\Theta$ is a homomorphism, consider the equation (in $Z(H_n(q))$) defining the structure constants $\phi_{\mu, \nu}^\lambda(q,n)$ of $\mathrm{FH}_q$,
\[
K_{\mu,q} K_{\nu,q} = \sum_{\lambda} \phi_{\mu, \nu}^\lambda(q,t) K_{\lambda,q},
\]
and apply $\theta \circ \Phi_{n,q}$ to get
\[
X_{\mu} X_{\nu} = \sum_{\lambda} \phi_{\mu, \nu}^\lambda(1,n) X_{\lambda},
\]
which is the equation (in $Z(\mathbb{Z}S_n)$) that determines the structure constants $\phi_{\mu,\nu}^\lambda(t)$ of $\FH$. We conclude that $ \phi_{\mu, \nu}^\lambda(1,t) =  \phi_{\mu, \nu}^\lambda(t)$, and hence $\Theta$ is a homomorphism.
\end{proof}
\noindent
Proposition \ref{prop:diagram} can be thought of as ``lifting'' the statement that $H_n(q)$ is a $q$-deformation of $\mathbb{Z}S_n$ to the level of the Farahat-Higman algebra.
\newline \newline \noindent
We will need some multiplicative properties of the Geck-Rouquier basis.

\begin{theorem}[Francis-Wang, Theorem 1.1 \cite{FrancisWang}] \label{thm:geck_rouquier_filtration}
The structure constant $\phi_{\mu, \nu}^\lambda(q,t)$ satisfies the following properties:
\begin{enumerate}
\item $\phi_{\mu, \nu}^\lambda(q,t) = 0$ unless $|\mu| + |\nu| \leq |\lambda|$,
\item if $|\mu| + |\nu| = |\lambda|$, $\phi_{\mu, \nu}^\lambda(q,t)$ is independent of $t$.
\end{enumerate}
\end{theorem}
\noindent
The first part of Theorem \ref{thm:geck_rouquier_filtration} immediately implies that $\mathrm{FH}_q$ is a filtered algebra.
\begin{proposition}
There is a filtration 
\[
\mathrm{FH}_q = \bigcup_{i \geq 0} \mathcal{F}_{i,q}
\]
where $\mathcal{F}_{i,q}$ is the $\mathcal{R}[q,q^{-1}]$-span of $K_{\mu,q}$ with $|\mu| \leq i$.
\end{proposition}

\subsection{Jucys-Murphy elements for the Iwahori-Hecke algebra}
The Jucys-Murphy (``JM'') elements for $H_n(q)$ are defined to be
\[
L_i = \sum_{1 \leq j < i} q^{j-i} T_{(i,j)}.
\]
This recovers the JM elements for $S_n$ when $q=1$. Note that this is not the only convention for JM elements (see Chapter 3, Section 3 of \cite{Mathas}), but it is the version that will be convenient for us. The following proposition is well known.

\begin{proposition}[\cite{Mathas} Corollary 3.27]
The JM elements $L_1, L_2, \ldots, L_n$ pairwise commute. Additionally, if $f$ is a symmetric polynomial in $n$ variables, then the evaluation $f(L_1, L_2, \ldots, L_n)$ is a central element of $H_n(q)$.
\end{proposition}
\noindent
As a result, we may evaluate symmetric functions at the first $n$ JM elements to obtain a cetral element of $H_n(q)$.
\begin{definition}
For $n$ a positive integer, let $ev_n: \mathcal{R}[q,q^{-1}] \otimes_{\mathbb{Z}} \Lambda \to Z(H_n(q))$ defined by evaluating the variable $t$ of $\mathcal{R}$ at $n$, and evaluating symmetric functions in $\Lambda$ at the JM elements $L_1, L_2, \ldots, L_n$.
\end{definition}

\noindent
For example, we have a direct analogue of Theorem \ref{thm:elem_JM}.

\begin{theorem}[\cite{FrancisGraham} Proposition 7.4] \label{thm:elem_JM-eval}
In $H_n(q)$, we have the following equation:
\[
ev_n(e_r) = e_r(L_1, L_2, \cdots, L_n) = \sum_{\mu \vdash r} \Gamma_\mu,
\]
where $\Gamma_\mu$ is the Geck-Rouquier basis element corresponding to $\mu$ (taken to be zero if $|\mu| + l(\mu) > n$).
\end{theorem}
\noindent
(Note that in \cite{FrancisGraham}, they do not use reduced cycle types, so the above equation is written slightly differently.) Thus we may immediately define the following homomorphism, a direct analogue of $\Psi$ from Theorem \ref{thm:FH_lambda_iso}.
\begin{definition}
Let $\Psi_q: \mathcal{R}[q,q^{-1}] \otimes_{\mathbb{Z}}\Lambda \to \mathrm{FH}_q$ be the homomorphism of $\mathcal{R}[q,q^{-1}]$-algebras defined by \[
\Psi_q(e_n) = \sum_{\mu \vdash n} K_{\mu,q}.
\]
\end{definition}
\noindent
We will now work towards proving $\Psi_q$ is an isomorphism.

\begin{proposition} \label{prop:triangular_CD}
We have the following commutative diagram:
\[
\begin{tikzcd}
\mathcal{R}[q,q^{-1}] \otimes_{\mathbb{Z}} \Lambda    \arrow[r, "\Psi_q"]    \arrow[dr, "ev_n"]    &    \mathrm{FH}_q    \arrow[d, "\Phi_{n,q}"]\\
&    Z(H_n(q))
\end{tikzcd}
\]
\end{proposition}
\begin{proof}
Since all the maps in the diagram are homomorphisms, it suffices to check the diagram commutes on a generating set of $\mathcal{R}[q,q^{-1}] \otimes_{\mathbb{Z}} \Lambda$; we choose the elementary symmetric functions $e_r$. Then we have
\begin{eqnarray*}
ev_n(e_r) &=& e_r(L_1, L_2, \ldots, L_n) \\
&=& \sum_{\mu \vdash r} \Gamma_\mu \\
&=& \Phi_{n,q} \left( \sum_{\mu \vdash r} K_{\mu, q} \right) \\
&=& \Phi_{n,q} (\Psi_q(e_r)),
\end{eqnarray*}
where we used Theorem \ref{thm:elem_JM-eval}, and as in Proposition \ref{prop:specisalisation_hom}, $\Gamma_\mu$ is interpreted as zero $|\mu| + l(\mu) > n$.
\end{proof}

\begin{remark}
Just as in the case of the map $\Psi$ for the usual Farahat-Higman algebra, one may view the map $\Psi_q$ as being ``evaluation at JM elements'', where we leave the number $n$ of JM elements unspecified.
\end{remark}

\begin{lemma}
The map $\Psi_q$ respects the grading.
\end{lemma}
\begin{proof}
This follows from the fact that $\Psi_q(e_r)$ is a sum of $K_{\mu, q}$, each of which has filtration degree $|\mu| = r$.
\end{proof}

\begin{theorem}
The associated graded map $\mathrm{gr}(\Psi_q)$ is an isomorphism.
\end{theorem}
\begin{proof}
We explain how this follows from the paper of Francis and Graham \cite{FrancisGraham}. We need to check that $\mathrm{gr}(\Psi_q)$ is an isomorphism for each degree $k$. The domain has a $\mathcal{R}[q,q^{-1}]$-basis consisting of monomial symmetric functions $m_\mu$ indexed by partitions $\mu$ of size $k$. The codomain has a basis consisting of $K_{\nu, q}$ indexed by partitions $\nu$ of size $k$. Let $N^{(k)}$ be the matrix expressing $\mathrm{gr}(\Psi_q)$ with respect to these two bases. Thus $N^{(k)}$ is a matrix whose entries are a priori elements of $\mathcal{R}[q,q^{-1}]$ and which obeys
\[
\Psi_q(m_{\mu}) = \sum_{\nu} N_{\nu,\mu}^{(k)} K_{\nu, q} + \cdots,
\]
where the ellipsis indicates terms of lower filtration degree. Now let us apply $\Phi_{n,q}$, which by Proposition \ref{prop:triangular_CD} gives
\[
ev_n(m_{\mu}) =  \Phi_{n,q} \left(\sum_{\nu}N_{\nu,\mu}^{(k)} K_{\nu, q} + \cdots \right) = \sum_{\nu} \left(N_{\nu,\mu}^{(k)}|_{t=n}\right) \Gamma_{\nu} + \cdots,
\]
where the notation $|_{t=n}$ indicates evaluation of the variable $t$ of $\mathcal{R}[q,q^{-1}]$ at $n$.
\newline \newline \noindent
Note that $\left(N_{\nu,\mu}^{(k)}|_{t=n}\right)$ is a matrix with entries $\mathbb{Z}[q,q^{-1}]$. Let us assume $n \geq 2k$, so that every partition $\nu$ of size $k$ obeys $|\nu| + l(\nu) \leq n$, making these $\Gamma_{\nu}$ linearly independent. In the introduction of \cite{FrancisGraham}, matrices $M^{(k)}(n)$ are defined which are identical to $\left(N_{\nu,\mu}^{(k)}|_{t=n}\right)$ for $n \geq 2k$. It is then explained that these matrices $M^{(k)}(n)$ are actually independent of $n$, provided that $n \geq 2k$; this result follows from Theorem 3.2 of \cite{Mathas_Stability}. In particular, for any integer $m$, the coefficient of $q^m$ in $N_{\nu,\mu}^{(k)} \in \mathcal{R}[q,q^{-1}]$ is given by a polynomial in the variable $t$, whose value at each $n \geq 2k$ is the same. Since this infinite set is Zariski dense, we conclude that actually the coefficient of $q^m$ in $N_{\nu,\mu}^{(k)}$ is independent of $t$, and hence that $N_{\nu,\mu}^{(k)} \in \mathbb{Z}[q,q^{-1}]$.
\newline \newline \noindent
Finally, Theorem 7.1 of \cite{FrancisGraham} asserts that $M^{(k)}(n)$ is invertible over $\mathbb{Z}[q,q^{-1}]$, where $n \geq 2k$. We conclude that the restriction of $\mathrm{gr}(\Psi_q)$ to degree $k$ is described by an invertible matrix, and therefore that the restriction of $\mathrm{gr}(\Psi_q)$ to degree $k$ is an isomorphism. Hence $\mathrm{gr}(\Psi_q)$ is an isomorphism.
\end{proof}

\begin{remark}
When $n < 2k$, the matrix $M^{(k)}(n)$ defined in \cite{FrancisGraham} is obtained by omitting the rows of $N^{(k)} = M^{(k)}(2k)$ corresponding to partitions $\nu$ with $|\nu| + l(\nu) > n$. The omitted entries are precisely the coefficients of those $\Gamma_{\nu}$ that are zero in $H_n(q)$.
\end{remark}

\begin{theorem}
We have that $\Psi_q$ is an isomorphism. In particular, $\mathrm{FH}_q \cong \mathcal{R}[q,q^{-1}] \otimes_{\mathbb{Z}} \Lambda$.
\end{theorem}
\begin{proof}
Both injectivity and surjectivity follow from the corresponding properties of the associated graded version of $\Psi_q$.
\end{proof}

\noindent
Noting that $\Lambda$ is generated by the elementary symmetric functions $e_r$ we immediately obtain a proof of a Conjecture of Francis and Wang.
\begin{corollary}[Conjecture 3.2  \cite{FrancisWang}]
The algebra $\mathrm{FH}_q$ is generated (over $\mathcal{R}[q,q^{-1}]$) by the elements
\[
\Psi_q(e_r) = \sum_{|\lambda| = r} K_{\lambda,q}.
\]
\end{corollary}

\section{Applications to Representation Theory} \label{section:applications}
\noindent
In this section we give a formula for the scalars by which the Geck-Rouquier basis acts on irreducible representations of $H_n(q)$ in terms of contents of partitions. Along the way we explain how the theory we have developed can be used to determine the blocks of Iwahori-Hecke algebras. This parallels the original work of Farahat and Higman, whose motivation was to give a simplified proof of the Nakayama Conjecture about $p$-blocks of symmetric group representations.
\newline \newline \noindent
First we recall several concepts related to partitions necessary to state the results: the $e$-core of a partition, and the contents of a partition.

\begin{definition}
If $\lambda$ is a partition, a \emph{border strip} of $\lambda$ of size $e$ is a subset $D$ of the boxes in the diagram of $\lambda$ obeying the following conditions:
\begin{itemize}
\item $D$ consists of $e$ boxes,
\item $\lambda \backslash D$ is the diagram of a partition,
\item $D$ is connected (boxes sharing an edge are considered adjacent, boxes sharing only a vertex are not),
\item $D$ does not contain a $2 \times 2$ square of boxes.
\end{itemize}
The \emph{$e$-core} of $\lambda$ is the partition obtained by successively removing border strips of size $e$ until it is no longer possible to do so.
\end{definition}
\noindent
It is well known that any choice of iteratively removing border strips of size $e$ results in the same $e$-core, so the $e$-core of $\lambda$ is well defined (see \cite{Macdonald} Chapter 1.1, Example 8).
\begin{example}
There are two ways of removing border strip of size 4 from the partition $(5,3,2)$, shown below. The black squares indicate the first border strip to be removed, while the grey squares indicate the second border strip.

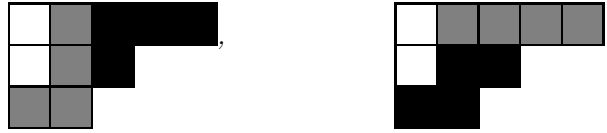
\begin{figure}[H]
\centering
\ytableausetup{nosmalltableaux}
\begin{ytableau}
*(white) & *(gray) & *(black) & *(black) & *(black)\\
*(white) & *(gray) & *(black) \\
*(gray) & *(gray)
\end{ytableau},
\hspace{20mm}
\begin{ytableau}
*(white) & *(gray) & *(gray) & *(gray) & *(gray)\\
*(white) & *(black) & *(black) \\
*(black) & *(black)
\end{ytableau}
\end{figure}
\noindent
In both cases, the remaining white squares form the diagram of the partition $(1,1)$ which is the $4$-core of $(5,3,2)$.
\end{example}

\begin{definition}
Suppose that $\lambda$ is a partition of size $n$. A \emph{standard Young tableau} of shape $\lambda$ is a bijective labelling of the boxes of the Young diagram of $\lambda$ with the numbers $1,2,\ldots,n$ such that the labels are increasing in each row and column (we use English notation for partitions). The \emph{content} of a box in row $i$ and column $j$ is $j-i$. If $T$ is a standard Young tableau, we write $c_T(m)$ for the content of the box containing the label $m$. 
\end{definition}

\begin{example}
Here are two standard Young tableaux of shape $(4,2,1)$:

\begin{figure}[H]
\centering
\ytableausetup{nosmalltableaux}
\ytableausetup{centertableaux}

$T_1$ = \begin{ytableau}
1 & 2 & 6 & 8\\
3 & 4 & 7 \\
5
\end{ytableau},
\hspace{20mm}
$T_2$ = \begin{ytableau}
1 & 4 & 6 & 7 \\
2 & 5 & 8 \\
3
\end{ytableau}
\end{figure}
\noindent
So we have $c_{T_1}(5) = 1-3 = -2$, while $c_{T_2}(5) = 2-2 = 0$. In both cases the multiset of the content of all boxes is $\{-2,-1,0,0,1,1,2,3\}$, because the contents are determined by the boxes in the diagram, and the labelling data of the tableau only prescribes the order.
\end{example}
\begin{definition}
If $m$ is an integer, we define the \emph{$q$-number} $[m]_q = 1 + q + \cdots + q^{m-1}$ if $m \geq 0$, and $[m]_q = -q^{-1} - q^{-2} - \cdots - q^{m}$ if $m < 0$. For $q \neq 1$, we have $[m]_q = \frac{q^m-1}{q-1}$.
\end{definition}
\begin{definition}
For a partition $\lambda$, let us write $\mathrm{cont}(\lambda)$ for the multiset of contents of boxes in the diagram of $\lambda$. We also write $\mathrm{cont}_q(\lambda)$ for the multiset of $[c]_q$ as $c$ ranges across $\mathrm{cont}(\lambda)$. We call $\mathrm{cont}_q(\lambda)$ the \emph{$q$-contents} of $\lambda$.
\end{definition}
\noindent
Recall that the characteristic zero irreducible representations of the symmetric group $S_n$ are the \emph{Specht modules} $S^\lambda$, indexed by partitions $\lambda$ of size $n$. The Specht modules are defined over $\mathbb{Z}$, and so one may base change to an arbitrary field $\mathbb{F}$. Hence we may view the Specht modules as modules for the algebra $\mathbb{F}S_n$, which may not be semisimple. This raises the question of when two Specht modules are in the same block (as representations of $\mathbb{F}S_n$). An answer is given by the Nakayama Conjecture (originally proved by Brauer and Robinson \cite{BrauerRobinson}), which asserts the following. Suppose that the characteristic of $\mathbb{F}$ is $p$. Then two Specht modules $S^\lambda, S^\mu$ are in the same block of $\mathbb{F}S_n$ if and only $\lambda$ and $\mu$ have the same $p$-core. As we will see, this result generalises neatly to $H_n(q)$, and we can give a new proof of this fact with using the $q$-Farahat-Higman algebra (see Corollary \ref{cor:Nakayama_conj} below). We now recall some facts about representations of $H_n(q)$; for details see \cite{Mathas}.
\newline \newline \noindent
For each partition of $\lambda$ of $n$, there is a $H_n(q)$-module $S^{\lambda}$ called the \emph{Specht module} indexed by $\lambda$. (If $q=1$ so that $H_n(q) = \mathbb{Z}S_n$, these Specht modules coincide with the ones for the symmetric group $S_n$.) Each Specht module for $H_n(q)$ is free as a $\mathbb{Z}[q,q^{-1}]$-module. Let $\mathbb{F}$ be a field with a distinguished nonzero element $\bar{q}$. Then $\mathbb{F}$ is a $\mathbb{Z}[q,q^{-1}]$-algebra via the homomorphism $\varphi: \mathbb{Z}[q,q^{-1}] \to \mathbb{F}$ defined by $q \mapsto \bar{q}$. We write 
\[
H_{n,{\mathbb{F}}}(\bar{q}) = \mathbb{F} \otimes_{\mathbb{Z}[q,q^{-1}]} H_n(q), \hspace{5mm} S_{\mathbb{F}}^{\lambda} = \mathbb{F} \otimes_{\mathbb{Z}[q,q^{-1}]} S^{\lambda}.
\]
Then $H_{n,{\mathbb{F}}}(\bar{q})$ is an $\mathbb{F}$-algebra, and $S_{\mathbb{F}}^{\lambda}$ is a module for this algebra.

\begin{definition}
If $\mathbb{F}$ is a field with a distinguished element $\bar{q}$, we define the \emph{quantum characteristic} to be the least positive integer $e$ such that $[e]_{\bar{q}} = 0$ as an element of $\mathbb{F}$:
\[
e = \min \left\{ m \in \mathbb{Z}_{>0} \mid [m]_{\bar{q}} = 0 \right\},
\]
or $e = \infty$ if there is no $m \in \mathbb{Z}_{>0}$ with $[m]_{\bar{q}} = 0$.
\end{definition}
\noindent
So for example, if $\bar{q}=1$, then $[m]_{\bar{q}} = m$, and so $e = p$ if $\mathrm{char}(\mathbb{F}) = p > 0$, while $e = \infty$ if $\mathrm{char}(\mathbb{F}) = 0$. The following lemma is then a trivial calculation.

\begin{lemma} \label{lemma:q_number_periodicity}
Let $m_1, m_2 \in \mathbb{Z}$ and $e$ the quantum characteristic. If $e = \infty$, then $[m_1]_{\bar{q}} = [m_2]_{\bar{q}}$ if and only if $m_1 = m_2$. Otherwise $[m_1]_{\bar{q}} = [m_2]_{\bar{q}}$ if and only if $m_1$ is congruent to $m_2$ modulo $e$.
\end{lemma}

\noindent
In Corollary 3.44 of \cite{Mathas}, it is explained that if $e > n$, then $H_{n, \mathbb{F}}(\bar{q})$ is semisimple and the Specht modules $S_{\mathbb{F}}^{\lambda}$ are precisely the irreducible representations of the Iwhaori-Hecke algebra. If $e \leq n$, $H_{m, \mathbb{F}}(\bar{q})$ not semisimple, and the irreducible representations are given by certain quotients of a certain subset of the $S_{\mathbb{F}}^{\lambda}$. The question of which Specht modules lie in the same block amounts to identifying when two Specht modules have the same central character. We now explain how the JM evaluation allows us to extract this information.

\begin{proposition}[\cite{Mathas} Proposition 3.37]
There is a basis $v_T$ of $S_{\mathbb{F}}^\lambda$ indexed by standard Young tableaux $T$ of shape $\lambda$, such that the action of each $L_i$ is lower triangular. Moreover, we have
\[
L_i v_T = [c_T(i)]_{\bar{q}} v_T + \cdots,
\]
where the ellipsis indicates higher-order terms. We refer to this basis as the \emph{Gelfand-Zetlin} (``GZ'') basis.
\end{proposition}

\begin{proposition} \label{prop:scalar_action}
Central elements of $H_{n,\mathbb{F}}(q)$ act on Specht modules by scalar multiplication.
\end{proposition}
\begin{proof}
The case for an arbitrary field $\mathbb{F}$ follows from the case of $\mathbb{Z}[q,q^{-1}]$ by base change. We may prove that case by embedding $\mathbb{Z}[q,q^{-1}]$ in its field of fractions $\mathbb{Q}(q)$, where the distinguished element $\bar{q}$ is the generating variable $q$. In that case, the quantum characteristic is $e = \infty$, so the Specht modules are irreducible, and the result follows from Schur's Lemma.
\end{proof}

\begin{theorem} \label{thm:Hecke_content_evaluation}
Suppose that $f \in \mathcal{R}[q,q^{-1}] \otimes_{\mathbb{Z}} \Lambda$, so that $ev_n(f)$ is a central element of $H_n(q)$. The scalar by which $ev_n(f)$ acts on the Specht module $S^\lambda$ is equal to $f$ evaluated at the $\bar{q}$-contents of $\lambda$ with the integer-valued polynomial variable $t$ set to $n$.
\end{theorem}

\begin{proof}
Consider the action of $ev_n(f)$ on a GZ basis vector $v_T$. We have $L_i v_T = [c_T(i)]_{\bar{q}} v_T + \cdots$, where the ellipsis indicates higher-order terms. Thus
\[
ev_n(f) v_T = f([c_T(1)]_{\bar{q}}, [c_T(2)]_{\bar{q}}, \ldots, [c_T(n)]_{\bar{q}}) |_{t=n} v_T + \cdots.
\]
However, by Proposition \ref{prop:scalar_action} we know that $ev_n(f)$ acts by scalar multiplication on $S^\lambda$ and hence on $v_T$. This means that the left hand side must be a multiple of $v_T$, and hence that the higher-order terms must vanish. In particular, the scalar we have multiplied $v_T$ by is
\[
f([c_T(1)]_{\bar{q}}, [c_T(2)]_{\bar{q}}, \ldots, [c_T(n)]_{\bar{q}}) |_{t=n} = f(\mathrm{cont}_{\bar{q}}(\lambda)) |_{t=n},
\]
and does not depend on which standard Young tableau $T$ was used for the following reason. Since $f$ is a symmetric function, the order of the contents as variables does not matter, and the multiset of contents is determined by the partition $\lambda$.
\end{proof}

\begin{lemma} \label{lemma:q-cont_core}
Two partitions $\lambda$ and $\mu$ have the same $e$-core if and only if $\mathrm{cont}_{\bar{q}}(\lambda) = \mathrm{cont}_{\bar{q}}(\mu)$ (where the elements of the multisets are contained in $\mathbb{F}$).
\end{lemma}
\begin{proof}
We modify the argument in \cite{Macdonald} Chapter 1.1, Example 11 by considering the $\bar{q}$-content polynomial
\[
c_{\lambda, \bar{q}}(x) = \prod_{P \in \mathrm{cont}(\lambda)}(x + [P]_{\bar{q}}).
\]
We briefly explain the steps.
\newline \newline \noindent
Suppose that $|\lambda| = n$. We pad $\lambda$ by trailing zeros so that it has a total of $n$ entries. Observing that the contents of the boxes in the $i$-th row of $\lambda$ are $1-i, 2-i, \ldots, \lambda_i-i$ and that $[m]_{\bar{q}} - 1 = \bar{q}[m-1]_{\bar{q}}$ we find that
\[
\frac{c_{\lambda, \bar{q}}(x)}{c_{\lambda, \bar{q}}(\bar{q}x-1)} =
\prod_{P \in \mathrm{cont}(\lambda)} \frac{x + [P]_{\bar{q}}}{\bar{q}x-1 + [P]_{\bar{q}}}
=
\prod_{P \in \mathrm{cont}(\lambda)} \frac{x + [P]_{\bar{q}}}{\bar{q}(x + [P-1]_{\bar{q}})}
=
(\bar{q})^{-n} \prod_{i=1}^{n}\frac{x+ [\lambda_i-i]_{\bar{q}}}{x + [-i]_{\bar{q}}}
\]
where in the last step we used the fact the terms for each row of $\lambda$ telescope. Now we may multiply by $\prod_{i=1}^n (x + [-i]_{\bar{q}})$ (which does not depend on $\lambda$) to conclude that $c_{\lambda, \bar{q}}(x)$ determines $\prod_{i=1}^{n} (x+ [\lambda_i-i]_{\bar{q}})$ and hence determines the multiset $[\lambda_i-i]_{\bar{q}}$ by unique factorisation in $\mathbb{F}[x]$. By Lemma \ref{lemma:q_number_periodicity}, this is equivalent to knowing the numbers $\lambda_i-i$ modulo $e$. By \cite{Macdonald} Chapter 1.1, Example 8, this is determines the $e$-core of $\lambda$. So the $\bar{q}$-contents of $\lambda$ determines $c_{\lambda, \bar{q}}(x)$, hence the left hand side of the above equation, hence also the $e$-core of $\lambda$
\newline \newline \noindent
Now we must show that the $e$-core determines the $\bar{q}$-contents of $\lambda$. For this we notice that if $\mu$ is obtained from $\lambda$ by removing a border strip of length $e$, then the contents of the boxes in the border strip $\lambda / \mu$ attain each residue class modulo $e$ exactly once. So we have
\[
c_{\lambda, \bar{q}}(x) = c_{\mu, \bar{q}}(x) \prod_{i=1}^e (x + [i]_{\bar{q}}),
\]
which is an equation in the unique factorisation domain $\mathbb{F}[x]$. Hence if $\tilde{\lambda}$ is the $e$-core of $\lambda$, we have
\[
c_{\lambda, \bar{q}}(x) = c_{\tilde{\lambda}, \bar{q}}(x) \left(\prod_{i=1}^e (x + [i]_{\bar{q}}) \right)^{\frac{|\lambda| - |\tilde{\lambda}|}{e}}.
\]
We conclude that if $\tilde{\lambda}$ (and the size $|\lambda|$) are known, so is $c_{\lambda, \bar{q}}(x)$, and hence the $q$-contents of $\lambda$.
\end{proof}

\begin{corollary}[\cite{Mathas} Corollary 5.38] \label{cor:Nakayama_conj}
Two Specht modules $S_{\mathbb{F}}^{\lambda}$ and $S_{\mathbb{F}}^{\mu}$ are in the same block if and only if the partitions $\lambda$ and $\mu$ have the same $e$-core.
\end{corollary}
\begin{proof}
Being the same block is equivalent to having the same central character, i.e. that every element of $Z(H_{\mathbb{F},n}(\bar{q}))$ should act by the same scalar on $S_{\mathbb{F}}^{\lambda}$ and $S_{\mathbb{F}}^{\mu}$. The centre is spanned over $\mathbb{F}$ by $ev_n(\Lambda)$, and $\Lambda$ is generated by the elementary symmetric functions $e_r$, so it is equivalent that each $ev_n(e_r)$ should act by the same scalar on $S_{\mathbb{F}}^{\lambda}$ as on $S_{\mathbb{F}}^{\mu}$. This in turn is equivalent to the following equality of generating functions:
\[
\sum_{r \geq 0} e_r(\mathrm{cont}_{\bar{q}}(\lambda)) x^r = \sum_{r \geq 0} e_r(\mathrm{cont}_{\bar{q}}(\mu)) x^r.
\]
Note that these power series are actually finite (the terms for $r>n$ are zero), and we have the following factorisation:
\[
\sum_{r \geq 0} e_r(\mathrm{cont}_{\bar{q}}(\lambda)) x^r = \prod_{P \in \mathrm{cont}(\lambda)} (1 + [P]_{\bar{q}}x),
\]
and similarly for $\mu$. Since $\mathbb{F}[x]$ is a unique factorisation domain, the polynomials for $\mu$ and $\nu$ coincide when the factors coincide. This is equivalent to $\mathrm{cont}_{\bar{q}}(\lambda)$ and $\mathrm{cont}_{\bar{q}}(\mu)$ coinciding as multisets (where the elements are contained in $\mathbb{F}$). But since $\lambda$ and $\mu$ have the same size, Lemma \ref{lemma:q-cont_core} implies this condition is equivalent to $\lambda$ and $\mu$ having the same $e$-cores.
\end{proof}
\noindent
Finally, we present a result on the action of the Geck-Rouquier basis on Specht modules. For symmetric groups, there is the notion of a \emph{class symmetric function} \cite{CorteelGoupilSchaeffer}. These are elements of $\mathcal{R} \otimes_{\mathbb{Z}} \Lambda$, whose evaluation at $\mathrm{cont}(\lambda)$ and $t=n$ gives the scalar by which a conjugacy-class sum (which is a central element) acts on the Specht module indexed by $\lambda$. The connection of these to the work of Farahat and Higman is explained in Section 3 of \cite{Ryba}. We construct analogous symmetric functions, replacing symmetric groups by Iwahori-Hecke algebras, conjugacy-class sums by the Geck-Rouquier basis elements, and contents by $q$-contents.


\begin{definition}
Let $f_{\mu} = \Psi_q^{-1}(K_{\mu, q})$, which is an element of $\mathcal{R}[q,q^{-1}] \otimes_{\mathbb{Z}} \Lambda$.
\end{definition}

\begin{proposition} \label{prop:Hecke_char_sym_fn}
Let $\lambda$ be a partition of $n$, and assume that $|\mu| + l(\mu) \leq n$. The scalar by which the Geck-Rouquier basis element $\Gamma_{\mu}$ acts on the Specht module labelled by $\lambda$ is equal to the evaluation of $f_\mu$ at $\mathrm{cont}_{\bar{q}}(\lambda)$ and $t=n$.
\end{proposition}
\begin{proof}
We have that 
\[
\Gamma_\mu = \Phi_{n,q}(K_{\mu, q}) = \Phi_{n,q}(\Psi_q(f_\mu)) = ev_n(f_\mu).
\]
Now we apply Theorem \ref{thm:Hecke_content_evaluation} to obtain the scalar by which $\Gamma_\mu$ acts on the Specht module indexed by $\lambda$.
\end{proof}

\begin{example}
We have the following equations:
\begin{eqnarray*}
\Psi_q(e_1) &=& \Gamma_{(1)} \\
\Psi_q(e_2) &=& \Gamma_{(1,1)} + \Gamma_{(2)} \\
\Psi_q(e_1^2) &=& (q+q^{-1})\Gamma_{(1,1)} + (q^2 + 1 + q^{-2}) \Gamma_{(2)} + (q-1)(n-1) \Gamma_{(1)} + q{n \choose 2}
\end{eqnarray*}
where the first two equations are from Theorem \ref{thm:elem_JM-eval} and the last equation follows from Example \ref{ex:Hecke_example}. This system of equations gives
\begin{eqnarray*}
f_{(1)} &=& e_1 \\
f_{(2)} &=& e_1^2 - (q+q^{-1})e_2 - (q-1)(n-1)e_1 - q{n \choose 2} \\
f_{(1,1)} &=& (q+1+q^{-1})e_2 - e_1^2 + (q-1)(n-1)e_1 + q{n \choose 2}.
\end{eqnarray*}
Together with Proposition \ref{prop:Hecke_char_sym_fn}, this gives formulae for the action of the Geck-Rouquier basis elements $\Gamma_{(1)}$, $\Gamma_{(2)}$ and $\Gamma_{(1,1)}$ on Specht modules. When we set $q=1$, we recover Example 3.10 of \cite{Ryba} for the symmetric groups.
\end{example}
\bibliographystyle{alpha}
\bibliography{ref.bib}

\begin{thebibliography}{Mat99b}

\bibitem[BR47]{BrauerRobinson}
R.~Brauer and G~de~B Robinson.
\newblock {On a conjecture by Nakayama}.
\newblock {\em Trans. Roy. Soc. Canada. Sect. III}, 41:20--25, 1947.

\bibitem[CGS04]{CorteelGoupilSchaeffer}
Sylvie Corteel, Alain Goupil, and Gilles Schaeffer.
\newblock Content evaluation and class symmetric functions.
\newblock {\em Advances in Mathematics}, 188(2):315--336, 2004.

\bibitem[FG06]{FrancisGraham}
Andrew~R Francis and John~J Graham.
\newblock {Centres of Hecke algebras: the Dipper--James conjecture}.
\newblock {\em Journal of algebra}, 306(1):244--267, 2006.

\bibitem[FH59]{FarahatHigman}
H.~K. Farahat and G.~Higman.
\newblock The centres of symmetric group rings.
\newblock {\em Proceedings of the Royal Society of London. Series A,
  Mathematical and Physical Sciences}, 250(1261):212--221, 1959.

\bibitem[FW09]{FrancisWang}
Andrew Francis and Weiqiang Wang.
\newblock {The centers of Iwahori-Hecke algebras are filtered}.
\newblock {\em Representation Theory, Comtemporary Mathematics}, 478:29--38,
  2009.

\bibitem[GR97]{GeckRouquier}
Meinolf Geck and Rapha{\"e}l Rouquier.
\newblock {Centers and simple modules for Iwahori-Hecke algebras}.
\newblock In {\em Finite Reductive Groups: Related Structures and
  Representations}, pages 251--272. Springer, 1997.

\bibitem[IK01]{IvanovKerov}
VN~Ivanov and SV~Kerov.
\newblock The algebra of conjugacy classes in symmetric groups and partial
  permutations.
\newblock {\em Journal of Mathematical Sciences}, 107(5):4212--4230, 2001.

\bibitem[Juc74]{Jucys}
A-AA Jucys.
\newblock Symmetric polynomials and the center of the symmetric group ring.
\newblock {\em Reports on Mathematical Physics}, 5(1):107--112, 1974.

\bibitem[KR21]{KannanRyba}
Arun~S Kannan and Christopher Ryba.
\newblock Stable centres ii: Finite classical groups.
\newblock {\em arXiv preprint arXiv:2112.01467}, 2021.

\bibitem[Mac15]{Macdonald}
I.~G. Macdonald.
\newblock {\em Symmetric functions and {H}all polynomials}.
\newblock Oxford Classic Texts in the Physical Sciences. The Clarendon Press,
  Oxford University Press, New York, second edition, 2015.
\newblock With contribution by A. V. Zelevinsky and a foreword by Richard
  Stanley, Reprint of the 2008 paperback edition [ MR1354144].

\bibitem[Mat99a]{Mathas}
Andrew Mathas.
\newblock {\em Iwahori-Hecke algebras and Schur algebras of the symmetric
  group}, volume~15.
\newblock American Mathematical Society Providence, RI, 1999.

\bibitem[Mat99b]{Mathas_Stability}
Andrew Mathas.
\newblock Murphy operators and the centre of the iwahori-hecke algebras of type
  a.
\newblock {\em Journal of Algebraic Combinatorics}, 9(3):295--313, 1999.

\bibitem[Mur83]{Murphy}
GE~Murphy.
\newblock {The idempotents of the symmetric group and Nakayama's conjecture}.
\newblock {\em Journal of Algebra}, 81(1):258--265, 1983.

\bibitem[Mé10]{Meliot_Hecke}
Pierre-Loïc Méliot.
\newblock {Products of Geck-Rouquier conjugacy classes and the Hecke algebra of
  composed permutations}.
\newblock {\em {Discrete Mathematics \& Theoretical Computer Science}}, {DMTCS
  Proceedings vol. AN, 22nd International Conference on Formal Power Series and
  Algebraic Combinatorics (FPSAC 2010)}, January 2010.

\bibitem[Ryb21]{Ryba}
Christopher Ryba.
\newblock Stable centres i: Wreath products.
\newblock {\em arXiv preprint arXiv:2107.03752}, 2021.

\end{thebibliography}

\end{document}